\documentclass[12pt,leqno]{amsart} 

\usepackage{amsmath}
\usepackage{amssymb}
\usepackage{amsthm} 
\usepackage[pagebackref]{hyperref} 
\hypersetup{
	colorlinks=true, 
	linkcolor=blue,  
	citecolor=blue, 
}
\usepackage{esint} 
\usepackage{enumerate} 
\usepackage[shortlabels]{enumitem}
\usepackage{mathtools}
\usepackage{comment} 
\usepackage[english]{babel}
\usepackage[dvipsnames]{xcolor}
\usepackage[a4paper, margin={3cm}]{geometry}

\numberwithin{equation}{section}

\newtheorem{lemma}{Lemma}[section]
\newtheorem{prop}[lemma]{Proposition}
\newtheorem{thm}[lemma]{Theorem}
\newtheorem{cor}[lemma]{Corollary}
\theoremstyle{definition}
\newtheorem{rmk}[lemma]{\bf Remark}
\newtheorem{defn}[lemma]{\bf Definition}

\newcommand{\N}{\mathbb{N}}

\newcommand{\R}{\mathbb{R}}

\renewcommand{\epsilon}{\varepsilon}
\newcommand{\ra}{\rightarrow}
\newcommand{\da}{\downarrow}

\newcommand{\HD}{\mathcal{H}}
\newcommand{\norm}[1]{\left \lVert #1 \right \rVert}
\newcommand{\abs}[1]{\left \lvert #1 \right \rvert}
\newcommand{\loc}{\mathrm{loc}}
\newcommand{\ol}{\overline}

\DeclareMathOperator{\divv}{\mathrm{div}}

\DeclareMathOperator{\dist}{\mathrm{dist}}

\DeclareMathOperator{\spt}{\mathrm{supp}}

\newcommand{\nocontentsline}[3]{}
\let\origcontentsline\addcontentsline
\newcommand\stoptoc{\let\addcontentsline\nocontentsline}
\newcommand\resumetoc{\let\addcontentsline\origcontentsline}

\allowdisplaybreaks

\title{Elliptic unique continuation below the Lipschitz threshold}
\author{Cole Jeznach}
\date{\today}
\addtocontents{toc}{\protect\setcounter{tocdepth}{1}}

\begin{document}

\begin{abstract}
	In this article, we investigate unique continuation principles for solutions $u$ of uniformly elliptic equations of the form $- \divv (A \nabla u) =0$ when $A$ is less regular than Lipschitz. For general matrices $A$, we prove that strong unique continuation holds provided  $A$ has modulus of continuity $\omega$ satisfying the Osgood condition, $\int_0^1 \omega(t)^{-1} dt = \infty$, plus some other mild hypotheses. Along with the counterexamples of \cite{Mandache-ceg}, this shows that the sharp condition on $A$ that guarantees unique continuation is essentially that $A$ is log-Lipschitz.
\end{abstract}

\maketitle

\section{Introduction}

Questions regarding unique continuation for solutions of linear PDEs are among the most fundamental. A linear partial differential operator $\mathcal{L}$, defined in a connected domain $\Omega \subset \R^n$, is said to satisfy the \textit{weak unique continuation principle} if every solution $u$ of $\mathcal{L}u = 0$ in $\Omega$ that vanishes on an open subset of $\Omega$ must vanish identically: $u \equiv 0$. If all solutions are  $L^2_{\loc}(\Omega)$, then we say that the operator $\mathcal{L}$ satisfies the \textit{strong unique continuation principle} provided that whenever $u$ vanishes to infinite order (in the $L^2$-sense) at some point $x \in \Omega$, i.e.,
\begin{align} \label{eqn:vanish_order}
	\lim_{r \ra 0} r^{-2N} \fint_{B_r(x)} u^2 \, dx = 0 \text { for all } N >0,
\end{align}
 then $u \equiv 0$. If the condition above fails for large $N$, then the \textit{vanishing order of $u$ at $x$} is defined as the supremum over all $ N \ge 0$ for which the limit in \eqref{eqn:vanish_order} vanishes. Understanding how the size and regularity of the coefficients defining $\mathcal{L}$ impact unique continuation remains a subtle and rich area of investigation, with many fundamental questions still open. Moreover, unique continuation principles (along with the techniques used to develop them) have found widespread applications, including in inverse problems \cite{KSU07, Haberman15, CR16}, harmonic measure and nodal geometry \cite{DF88, Logunov18-poly, Logunov18-lower, LMNN21, Tolsa23}, quantitative uniqueness results for PDE \cite{ESS03a, ESS03b, KSW15,LMNN25}, and Anderson localization \cite{BK05}.

The focus of this article is linear, symmetric, uniformly elliptic operators
\begin{align}\label{eqn:op}
	\mathcal{L} u  = -\divv(A \nabla  u )   , \, \, \text{ in } B_1 \subset \R^n.
\end{align}
Of course when $A = I$ is the identity, solutions of \eqref{eqn:op} are harmonic, and hence analytic; as such, $-\Delta$ satisfies the strong unique continuation principle. In his pioneering work \cite{Carleman-planar}, Carleman proved that strong unique continuation also holds for operators of the form \eqref{eqn:op} in two dimensions, provided that $A \in C^2(B_1)$. The method he introduced via weighted inequalities--now known as Carleman estimates-- has become a ubiquitous tool for proving unique continuation results across a wide range of PDEs, well beyond the elliptic setting. Since then, the role of the smoothness condition on $A$ has largely been clarified. When $n = 2$, strong unique continuation holds even for $A$ which are merely bounded measurable \cite{BN55,Schulz98}. For dimensions $n \ge 3$, improvements by \cite{Aronszajn,Agmon}, and finally \cite{GL86} showed that Lipschitz regularity of $A$ suffices to guarantee strong unique continuation. Such results are valid even when $\mathcal{L}$ has singular lower-order terms; if $\mathcal{L}u = - \divv(A \nabla u) + W \cdot \nabla u + V u$ with $A$ Lipschitz, $W \in L^{n}$, and $V \in L^{n/2}$ then strong unique continuation still holds \cite{JK85, Wolff92, KT01}. Moreover, these exponent thresholds are essentially sharp.

In the other direction, Miller showed in dimensions $n \ge 3$ \cite{Miller-ceg} that weak unique continuation for \eqref{eqn:op} may fail if we only assume $A$ is H\"{o}lder continuous (see also the earlier example of Plis in non-divergence form \cite{Plis-ceg}). Later, Mandache improved the counterexample of Miller and showed that there is a matrix $A$ which is H\"{o}lder continuous of every exponent $\alpha \in(0,1)$ but for which $\mathcal{L}$ does not satisfy the weak unique continuation principle. In fact, Mandache provided for each modulus of continuity $\omega:[0,1] \ra [0,\infty)$ satisfying
\begin{align}\label{eqn:non_osgood}
	\int_0^1 \dfrac{1}{\omega(t)} \, dt < \infty,
\end{align}
an associated symmetric, uniformly elliptic matrix $A$ with modulus of continuity $\omega$ for which weak unique continuation fails for the operator $\mathcal{L}$ \cite{Mandache-ceg}. Taking $\omega(t) = t \log(1/t)^p$ for $p > 1$ then gives the claimed matrix $A$. It is worth noting that the examples of both \cite{Miller-ceg} and \cite{Mandache-ceg} are anisotropic, i.e., $A(x)$ cannot be written as $A(x) = a(x)I$ for any scalar function $a$.

Unlike Carleman's, the methods of \cite{GL86, GL87} relied on monotonicity properties of the so-called \textit{Almgren frequency function} $N_u^A(r)$, which for $\mathcal{L} = -\Delta$ takes the form
\begin{align*}
	N_u^I(r) = \dfrac{r^2 \fint_{B_r} \abs{\nabla u}^2 \, dx }{\fint_{\partial B_r} u^2 \, d\sigma}.
\end{align*}
This frequency function, which measures in a scale-invariant way the dominant homogeneity of $u$ in $B_r$, is known to be almost-monotone and bounded in $r$ if $A$ is Lipschitz. Such properties of this frequency function and its relatives have proven to be essential in quantitative unique continuation and free boundary problems. Indeed, bounds on the frequency function allow one to deduce doubling inequalities for the measure $u^2 \,dx$ and obtain structural information on the infinitesimal behavior of the solution $u$ about the origin. It is impossible to list all such applications, so instead we point the reader to some of the recent results in quantitative unique continuation such as \cite{CNV15, NV17-annals, NV17, LM18-23, LM18, Logunov18-poly, Gallegos23, CLMM24,  Foster25}, as well as the references \cite{LM-notes, Spolaor22, SVG24} for further reading. See also the recent results of \cite{Davey25}, where the author reproves elliptic unique continuation results originally obtained via the Carleman method using instead monotonicity properties of the Almgren frequency.
 
\vspace{0.5em}

The purpose of this article is to study the monotonicity properties of the Almgren frequency (and associated unique continuation principles) when $A$ is less regular than Lipschitz. In this general anisotropic setting, we shorten the gap between the counterexamples for unique continuation provided by \cite{Mandache-ceg} and the sufficiency of the Lipschitz condition used in \cite{GL86} to prove strong unique continuation. Our main result applies to matrices $A$ which admit a modulus of continuity $\omega$ for which \eqref{eqn:non_osgood} fails (such $\omega$ are said to satisfy the \textit{Osgood} condition).
\begin{thm}\label{thm:main_anis}
	Let $u \not \equiv 0$ solve 
	\begin{align*}
		-\divv(A \nabla u) = 0 \text{ in } B_1 \subset \R^n,
	\end{align*}
	where $A(0) = I$, $A$ is symmetric and uniformly elliptic with constant $\lambda \in (0,1)$, and $A$ has modulus of continuity $\omega(t)$ for which $\omega(t)$ and $\phi(t) \coloneqq \omega(t)/t$ satisfy
	\begin{enumerate}
		\item $\int_0^1 \omega(t)^{-1} \, dt = \infty$, \label{cond:osgood}
		\item $\phi(st) \le C \phi(s) g(t)$ for some $C> 1$ and all $0 < s, t \le C^{-1}$, \label{cond:semi}
		\item $g \in L^1((0, 1); \R_+)$ satisfies $g(t) \le C g(\gamma t)$ for $0 < t <C^{-1}$, $\gamma \in (1/2, 1)$. \label{cond:int}
	\end{enumerate}
	Then the vanishing order of $u$ at $0$ does not exceed $N$, where $N$ is some value depending just on $n$, $\lambda$, the modulus of continuity $\omega$, and the frequency $N_u^A(1)$ of $u$ in $B_1$.
\end{thm}
See Definition \ref{defn:almgren} for the expression $N_u^A(1)$. In view of the examples of \cite{Mandache-ceg} and Theorem \ref{thm:main_anis}, the threshold for unique continuation is more accurately described as $A$ having \textit{Osgood} modulus of continuity $\omega$. Indeed in practice the conditions \eqref{cond:semi}, \eqref{cond:int} are likely satisfied by many moduli of continuity $\omega$ which satisfy \eqref{cond:osgood}, and Theorem \ref{thm:main_anis} applies to, for example $\omega(t) = t \log(1/t)$. Along with \cite{Mandache-ceg}, we then infer that if $A$ has modulus of continuity given by $\omega(t) = t \log(1/t)^p$, then unique continuation for $\mathcal{L}$ holds for $p =1$ but may fail for any choice of $p >1$. A natural context in which such a matrix $A$ appears is one where $\nabla A$ has mild singularities, e.g., if $\abs{\nabla A(x)} \le \abs{\log(\dist(x, \Gamma))}$ where $\Gamma$ is a smooth $(n-1)$ manifold. 

Finally, it is worth noting that since our methods give uniform bounds on the frequency function $N_u^A(r)$, then we can also obtain uniform doubling properties of solutions and apply the recent results of \cite{HW23-holder} to obtain the following:
\begin{thm}\label{thm:main_crit}
	Let $u \not \equiv 0$ solve
	\begin{align*}
		-\divv(A \nabla u) = 0 \text{ in } B_1 \subset \R^n,
	\end{align*}
	where $A$ is H\"{o}lder continuous of order $\alpha \in (0,1)$ and $A$ satisfies the hypothesis of Theorem \ref{thm:main_anis}. Then the critical set $\mathcal{C}(u) \coloneqq \{ x \in B_1 \, : \, \nabla u(x) = 0\}$ is $(n-2)$-rectifiable and has locally finite $(n-2)$-Minkowski content $B_1$.  
\end{thm}
The point of Theorem \ref{thm:main_crit} is that the solution $u$ is not assumed apriori to be doubling, such as in the main results of \cite{HW23-holder}. Theorem \ref{thm:main_crit} can be made more quantitative (using the more quantitative conclusion of Theorems \ref{thm:freq_bdd}), but we do not expect the estimates to be optimal. It is likely that an analogue Theorem \ref{thm:main_crit} holds for the nodal set $\mathcal{N}(u) \coloneqq \{ x \in B_1 \; : \; u(x) = 0\}$ of solutions as well (replacing the dimension $(n-2)$ with $(n-1)$), but since \cite{HW23-holder} do not study the size of the nodal set, we content ourselves just with stating the application to the size of critical sets.

\subsection{Discussion of the main results and outline of the paper}

Theorem \ref{thm:main_anis} is new in that it shows unique continuation in a class of operators without requiring estimates (or structure) on $\nabla A$. Indeed it is worth pointing out that unique continuation results have been shown for Lipschitz coefficients $A$ with jump discontinuities across a hypersurface \cite{LRL13, DCFLVW17, FVW22}, or degenerate-singular elliptic equations when $\nabla A$ has some homogeneous structure \cite{CS07, GSVG14, Yu17, SST20, EJSpreprint}. 

It remains an interesting open questions exactly how additional structure on the matrix $A$ could yield results better than Theorem \ref{thm:main_anis}. For example, in the isotropic setting ($A = aI$), if $a \in W^{1,p}(B_1)$ for $p > n$, then strong unique continuation holds as a consequence of the results in \cite{Wolff92, KT01}, since $u$ solves $-\divv (a \nabla u) = 0$ if and only if 
\begin{align}\label{eqn:iso_reform}
	- \Delta u & = \nabla (\log(a)) \cdot \nabla u.
\end{align}
In particular, for such operators, we know that Theorem \ref{thm:main_anis} is not sharp. On the other hand, there are no known counterexamples for strong unique continuation in the class of \textit{isotropic}, H\"{o}lder continuous coefficients $A = aI$, and so a large gap remains open. Better understanding the role of such structural assumptions could prove useful in the Calder\'{o}n problem for isotropic conductivities \cite{Haberman15, CR16}.

\vspace{0.5em}

The idea of the proof of \ref{thm:main_anis}, is to use the regularity of the coefficients on $A$ to control the size of the Almgren frequency function associated to a solution, $u$, from the scale $r$ to the scale $r(1-1/N)$, where $N = N_u^A(r)$ is the frequency of $u$ at the scale $r$. After repeating this many times, we show that in fact $N_u^A(r)$ is bounded for all sufficiently small $r$ which readily implies the bound on the vanishing order.

When $A$ has an Osgood modulus of continuity, we can directly smooth the coefficients $A$ and pass estimates on a solution $v$ to an equation with smooth coefficients directly to the solution $u$. The key estimate in this section is Lemma \ref{lem:dichot}, which is in some sense an integrated version of almost-monotonicity of $N_u^A(r)$. The Osgood condition is well-suited to this context, and guarantees slow enough growth from $r$ to $r(1-1/N)$ to ensure that $N_u^a(r)$ does not blow up at a finite radius $r'$; see Lemma \ref{lem:growth_smooth}.

\subsection{Conventions and organization of the paper}

The entirety of this paper is valid in $\R^n$ for $n \ge 2$. As is usual in analysis papers, we denote by $C > 1$ and $ 0 < c < 1$ large and small constants which may change from line to line, and only depend on the dimension $n$ unless otherwise specified. We use the symbol $A \lesssim_D B$ to mean that there is some constant $C > 1$ depending on $n$ and $D$ for which $A \le C B$, and similarly, $A \simeq_D B$ to mean that $A \lesssim_D B \lesssim_D A$.  Since it shall prove convenient in the proof of Lemma \ref{lem:almost_monoton}, we use the big O notation $O(A)$ to denote some real number whose absolute value is bounded by $CA$. More specifically, $A = B + O(D)$ means that $\abs{A - B} \lesssim D$. Finally, we sometimes use the phrase ``for $\epsilon$ sufficiently small'' to mean that some statement is true for all $0 < \epsilon < C^{-1}$ where $C >1 $ is a constant depending just on $n$. We also adopt the usual notation that a matrix $A$ is uniformly elliptic with constant $\lambda  >0$ provided that 
\begin{align*}
	\lambda \abs{\xi}^2 \le A \xi \cdot \xi, \qquad \abs{A \xi} \le \lambda^{-1} \abs{\xi}, \, \, \xi \in \R^n. 
\end{align*}
Moreover, we write $\abs{E}$ for the Lebesgue measure of a set $E \subset \R^n$, denote by $\sigma$ surface measure on $\partial B_r$, and adopt the average-integral notation
\begin{align*}
	\fint_E f \, d\mu = \dfrac{1}{\mu(E)} \int f \, d\mu.
\end{align*}
Often we will omit $dx$ or $d\sigma$ (Lebesgue and surface measure) from integral expressions when the domain of integration is $B_r$ or $\partial B_r$ respectively. 

In Section \ref{sec:anisotropic} we provide a proof of Theorem \ref{thm:main_anis}, and in the Appendix we list a simple proof quantitative stability for solutions to elliptic PDEs.

\stoptoc
\section*{Acknowledgements} 
C.J. was supported by the European Research Council (ERC) under the European Union's Horizon 2020 research and innovation programme (grant agreement 101018680), and by the National Science Foundation through NSF-DMS-2503326. He thanks S. Decio, M. Engelstein, S. Mayboroda, and E. Malinnikova for helpful conversations regarding the main result. He would also like to thank M. Engelstein for providing helpful comments on the draft of this paper, and S. Decio for explaining to him the results in \cite{MS69}.
\resumetoc

\section{Anisotropic equations}\label{sec:anisotropic}

In this section, we prove the main results, Theorem \ref{thm:main_anis}. The main tools in the proof consist of elementary boundedness results for functions satisfying a differential inequality (Lemma \ref{lem:growth_smooth}) and almost-monotonicity of the Almgren frequency function when $A$ is Lipschitz (Lemma \ref{lem:almost_monoton}). Through a careful approximation procedure, we combine them to arrive at the main result, which is Theorem \ref{thm:freq_bdd}.

First, let us introduce some notation and basic properties related to moduli of continuity.

\begin{defn}
	We say that $\omega:[0, 1] \ra [0, \infty)$ is a modulus of continuity if $\omega$ is a continuous, non-decreasing, concave function, and $\omega(0) = 0$. We further say that $\omega$ satisfies the \textbf{Osgood condition} provided that 
	\begin{align}\label{eqn:osgood}
		\int_0^1 \dfrac{dt}{\omega(t)} = \infty.
	\end{align}
\end{defn}

Associated to the modulus of continuity $\omega$, we introduce the function $\phi(s) = \omega(s)/s$ which has the following properties.
\begin{prop}\label{prop:mod}
Let $\omega:[0,1] \ra [0, \infty)$ be a modulus of continuity, and write $\omega(s) = s\phi(s)$, $\psi(s) = \phi(1/s)$. Then $\phi$ is decreasing, $\psi$ is increasing, and $1/(s\psi(s))$ is decreasing. 
\end{prop}

From the previous proposition, we know that $\lim_{s \ra 0^+} \phi(s)$ exists. Since the case when $\lim_{s \ra 0^+} \phi(s) < \infty$ corresponds to Lipschitz functions (which, from the work of \cite{GL86} is well-understood), we may as well restrict our attention to non-Lipschitz modululi $\omega$, i..e, the case when
\begin{align*}
	\lim_{s \ra 0^+} \phi(s) = + \infty.
\end{align*}
First, let us make the following elementary observations, which will be the main tool in the proof of Theorem \ref{thm:freq_bdd}.
\begin{lemma}\label{lem:growth_smooth}
	Let $\omega$ be an Osgood modulus of continuity, and let $f:[0,1] \ra [0, \infty]$ be a decreasing function (locally Lipschitz whenever it is finite) satisfying the differential inequality 
	\begin{align}\label{eqn:diff_ineq}
		f'(t) \ge - C_1 f(t) \psi(f(t)) g(t), \, \, \text{a.e. } t \in (0,1], \, \, f(t) < \infty,
	\end{align}
	for some nonnegative $g \in L^\infty_{\loc}(0,1)$ and some constant $C_1 > 0$. If $f(1) < \infty$, then $f$ is bounded on compact subsets of $(0, 1]$. If in addition $\int_0^1 g(t) \, dt < \infty$, then $f$ is bounded on $[0,1]$. 
\end{lemma}
\begin{proof}
	Recall the notation $\omega(s) = s \phi(s)$, $\psi(s) = \phi(1/s)$. Let $h(t) = \int_1^t 1/(s\psi(s)) \, ds$, so that $h'(t) = 1/(t\psi(t))$, and $h(t)$ is an increasing function. By assumption, we have 
	\begin{align*}
		\dfrac{d}{dt} h(f(t)) \ge - C_1 g(t),
	\end{align*}
	so that 
	\begin{align*}
		h(f(t))  \le h(f(1)) + C_1 \int_{t}^1 g(s) \, ds, \, \, t \in (0, 1).
	\end{align*}
	Since $g \in L^\infty_{\loc}(0,1)$, we see that $h\circ f$ is bounded on compact subsets of $(0,1]$. Recalling that $\omega$ satisfies the Osgood condition \eqref{eqn:osgood}, we know that 
	\begin{align*}
	\infty = 	\int_0^1 \dfrac{ds}{\omega (s)} = \int_1^\infty \dfrac{ds}{s\psi(s)},
	\end{align*}
	and thus $f(t)$ is bounded on compact subsets of $(0,1]$ as well. In addition this also shows that if $g \in L^1(0,1)$, then $h \circ f$ (and thus $f$) is bounded on $[0,1]$, which completes the proof.
\end{proof}

It is a simple Corollary of Lemma \ref{lem:growth_smooth} that if one instead assumes an integrated version of the differential inequality \eqref{eqn:diff_ineq}, then the same conclusion holds. The proof consists in simply replacing the function $f$ by its linear interpolant in the scales $[t_k(1-1/f(t_k)), t_k]$ for an appropriately chosen sequence $t_k \da 0$ and applying the Lemma above.

 \begin{cor}\label{cor:growth_disc}
 	Let $\omega$ be an Osgood modulus of continuity, and let $f:[0,1] \ra [0, \infty]$ be a decreasing function satisfying
 	\begin{align}\label{eqn:diff_ineq_disc}
 		f \left(  t \left( 1 - 1/f(t)  \right)   \right) \le f(t) + C_1 t \psi(f(t)) g(t), \, \, t \in (0,1), \, \, f (t) < \infty,
 	\end{align}
 	for some nonnegative $g \in L^\infty_{\loc}(0,1)$ and some constant $C_1 > 0$. If $g$ is doubling in the sense that 
 	\begin{align*}
 		g(s) \le C_1 g(\gamma s), \, \, \gamma \in (1/2, 1),
 	\end{align*}
 	and $f(1) < \infty$, then $f$ is bounded on compact subsets of $(0, 1]$. If in addition $\int_0^1 g(t) \, dt < \infty$, then $f$ is bounded on $[0,1]$. 
 \end{cor}
\begin{proof}
	 Let $N_0 \coloneqq f(1) < \infty$, $t_0 \coloneqq 1$, and define inductively the sequences $t_{k+1} \coloneqq t_k(1 - 1/ N_k)$, $N_{k+1} \coloneqq f(t_{k+1}) \le N_k + C_1 t_k \psi(N_k) g(t_k)$. Notice that we may assume $N_0 \ge 2$, otherwise we simply choose $t_0$ slightly smaller. Consider the function $h$, which on each interval $I_k \equiv [t_{k+1}, t_k]$ interpolates linearly between $N_{k+1}$ and $N_k$:
	\begin{align*}
		h(t) \equiv N_{k+1} + \dfrac{N_k - N_{k+1}}{t_k - t_{k+1}} (t - t_{k+1}), \, \, t \in I_k.
	\end{align*}
	Since $f$ is decreasing, we see that $h$ also is, and that $h(t) \ge N_k$ for $t \in I_k$. Moreover, we see that 
	\begin{align*}
		h'(t) = (N_k/t_k) (N_k - N_{k+1}) \ge - C_1 N_k \psi(N_k) g(t_k), \, \, t \in I_k.
	\end{align*}
	Using that $\psi$ is increasing and the doubling condition on $g$, the above implies readily that 
	\begin{align*}
		h'(t) \ge -C h(t) \psi(h(t)) g(t), \, \, t \in I_k
	\end{align*}
	and thus we may apply Lemma \ref{lem:growth_smooth} to the function $h$. Since $h(t_k) = f(t_k)$ and $f$ is decreasing, the boundedness properties of $h$ from that Lemma transfer directly to $f$, completing the proof of the Lemma. 
\end{proof}

	Our goal is to apply Corollary \ref{cor:growth_disc} to a suitable variant of the Almgren frequency function for a solution $u$ of the uniformly elliptic equation
	\begin{align}
		-\divv(A \nabla u) & = 0  \text{ in } B_2 \subset \R^n, \label{eqn:soln}
	\end{align}
	where $A$ has an Osgood modulus of continuity $\omega$. Unique continuation principles for solutions will then follow quickly from boundedness of this frequency function. The frequency $N_u^A(r)$, which we define now, measures in some sense the dominant homogeneity of $u$ inside the ball $B_r$.
	\begin{defn}\label{defn:almgren}
		If $u$ solves \eqref{eqn:soln} in $B_{R}$, we define the Almgren frequency
		\begin{align}
			N_{u}^A(r) \coloneqq \dfrac{r   \int_{B_r} \langle A \nabla u, \nabla u \rangle  \, dx }{\int_{\partial B_r} u^2 \mu  \, d\sigma } \equiv \dfrac{r D_u^A(r)}{H_u^A(r)}, \, \, 0 < r  < R,
		\end{align}
		where $\mu(x) \coloneqq \langle A(x) x/\abs{x}, x/\abs{x}\rangle$.
	\end{defn}

	When the matrix $A$ is Lipschitz, the Almgren frequency function was shown to be almost-monotone in \cite{GL86}. Since we need slightly more precise estimates, we provide a detailed proof in the following Lemma. In the end these estimates will be applied to solutions of $\divv( A_\epsilon \nabla v) = 0$ where $A_\epsilon = A * \eta_\epsilon$, $\eta$ is an approximate identity, and $\epsilon$ is a small parameter depending on the frequency of the solution $u$ at scale $r$. 
	\begin{lemma}\label{lem:almost_monoton}
		Let $u$ be a solution of \eqref{eqn:soln} in $B_R$, and suppose that $A$ is symmetric and uniformly elliptic with constant $\lambda >0$, and $A$ satisfies the smoothness estimates 
		\begin{align*}
			\abs{\nabla A(x)} \le M, \qquad \abs{A(x) - I } \le \delta, \, \, x \in B_R, 
		\end{align*}
		for some $ M, \delta >0$. Then we have the estimate 
		\begin{align*}
			\dfrac{d}{dr}  \log \left(   N_u^A(r) \right) \ge - C_1 ( M + \delta/r), \, \, r \in (0,R),
		\end{align*}
		for some constant $C_1 >0$ depending just on the dimension $n$ and $\lambda$. In addition, we have the estimate
		\begin{align*}
			\dfrac{d}{dr} \log( r^{-n+1} H_u^A(r)) & =  2 N_u^A(r)/r + e(r), \, \, r \in (0, R),
		\end{align*}
		with error $\abs{e(r)} \le C_1 ( M + \delta/r)$. 
	\end{lemma}
	\begin{proof}
		Recall from the definition of $N_u^A$, we have that $\mu(x) \equiv \langle A(x)x/\abs{x}, x/\abs{x}\rangle$. Define the vector field $\beta(x) \coloneqq A(x)x/\mu(x)$. Since the unit outer normal $\nu(x)$ to $\partial B_r$ is given by $\nu(x) = x/\abs{x}$, then we have by definition of $\mu$ that
		\begin{align}
			A(x) \nu(x) \equiv \mu(x) \beta(x)/r, \, \,  \beta(x) \cdot \nu(x) = r, \, \, x \in \partial B_r,  \label{eqn:beta_norm}
		\end{align}
		Straight-forward arguments using the ellipticity and smoothness conditions on $A$ give us the following estimates on $\mu$ and $\beta$:
		\begin{align}
			& \mu(x)  \simeq 1, \, \abs{\nabla \mu(x)} \lesssim M + \delta/\abs{x}, \, \abs{\mu(x)  -1} \lesssim \delta, \label{eqn:mu_est} \\
			& \abs{\beta(x) - x} \lesssim \delta \abs{x}, \, \abs{D\beta(x) - I} \lesssim  M \abs{x}  + \delta, \label{eqn:beta_est}
		\end{align}
		with implicit constants depending just on the dimension and $\lambda$. Of course in order to prove the estimates above, one uses that if $A \equiv I$, then $\mu(x) \equiv 1$ and $\beta(x) \equiv x$. 
		For convenience, write $N(r) = N_u^A(r)$, $D(r) = D_u^A(r)$ and $H(r) = H_u^A(r)$. One computes
		\begin{align}\label{eqn:logN'}
			\dfrac{d}{dr} \log(N(r)) & = \dfrac{1}{r}  + \dfrac{D'(r)}{D(r)} - \dfrac{H'(r)}{H(r)},
		\end{align}
		and we deal with each term separately.
		
		First, as for $D'(r)$ we have from \eqref{eqn:beta_norm} and the divergence theorem that
		\begin{align*}
			D'(r) & = \int_{\partial B_r} \langle A \nabla u, \nabla u \rangle \\
			& = r^{-1} \int_{\partial B_r} \langle A \nabla u, \nabla u \rangle \langle \beta, \nu \rangle  \\
			& = r^{-1} \int_{B_r} \divv( \langle A \nabla u, \nabla u \rangle \beta ) ,
		\end{align*}
		and so using the symmetry of $A$,
		\begin{align*}
			rD'(r) & = \int_{B_r} \divv(\beta) \langle A \nabla u, \nabla u \rangle + \int_{B_r} \sum_{i,j,k} \beta^k (A_{ij})_{x_k} u_{x_i} u_{x_j} + 2 \int_{B_r} \sum_{i,j,k} \beta^k  A_{ij} u_{x_i} u_{x_j x_k} \\
			& \coloneqq T_1 + T_2 + T_3.
		\end{align*}
		By courtesy of \eqref{eqn:beta_est} and the fact that $\mathrm{Trace}(I) = n$, we have
		\begin{align*}
			T_1 = n D(r) + O(M r + \delta) D(r) , \,\, T_2 = O(M r) D(r).
		\end{align*}
		As for $T_3$, we integrate by parts, use that $u$ solves \eqref{eqn:soln}, and again apply \eqref{eqn:beta_norm}, \eqref{eqn:beta_est} to obtain 
		\begin{align*}
			T_3 & = -2 \int_{B_r} \sum_{ijk} (\beta^k)_{x_j} A_{ij} u_{x_i} u_{x_k} + 2 \int_{\partial B_r} \langle \beta, \nabla u \rangle \langle A \nabla u, \nu \rangle \\
			& = -2 D(r) + O(Mr + \delta) D(r) + 2r \int_{\partial B_r} \mu^{-1} \langle A \nabla u, \nu \rangle^2 .
		\end{align*}
		Finally, we use once again that $u$ is a solution to rewrite
		\begin{align}\label{eqn:D}
			D(r) & = \int_{B_r} \divv(u A \nabla u)  = \int_{\partial B_r} u \langle A \nabla u, \nu\rangle. 
		\end{align}
		Combining all of the the previous equalities yields
		\begin{align}\label{eqn:D'}
			D'(r)/D(r) & = \dfrac{n-2}{r} + O(M + \delta/r) + 2 \dfrac{ \int_{\partial B_r} \mu^{-1} \langle A \nabla u, \nu \rangle^2  }{  \int_{\partial B_r} u \langle A \nabla u, \nu \rangle   }.
		\end{align} 
		
		Moving on to $H'(r)$, one easily computes with the estimate on $\mu$ from \eqref{eqn:mu_est} that 
		\begin{align*}
			H'(r) & = \dfrac{n-1}{r} H(r) + \int_{\partial B_r} \partial_\nu ( u^2 \mu) \\
			& = \dfrac{n-1}{r} H(r) + O(M + \delta/r) H(r) + 2\int_{\partial B_r} u \partial_\nu u \mu. 
		\end{align*}
		Next we notice that $A(x) \nu(x) - \mu(x) \nu(x) \equiv r^{-1} \mu(x) (\beta(x) - x)$ is a tangent vector field to $\partial B_r$ which satisfies
		\begin{align*}
			\abs{\divv_{\partial B_r} (A\nu - \mu \nu)} \lesssim M + \delta/r.
		\end{align*}
		Indeed the estimate above again comes from the estimates \eqref{eqn:mu_est}, \eqref{eqn:beta_est}. Thus we can apply the divergence theorem on $\partial B_r$ and use the symmetry of $A$ to obtain
		\begin{align*}
			\abs{ \int_{\partial B_r} 2 u \partial_\nu u \mu  - 2 u \langle A \nabla u, \nu \rangle } & = \abs{ \int_{\partial B_r} \langle \nabla (u^2), \mu \nu - A \nu \rangle   } \\
			& = \abs{ \int_{\partial B_r}  u^2 \divv_{\partial B_r}(\mu \nu - A \nu) } = O(M + \delta/r) H(r).
		\end{align*}
		Moreover, we can use this to rewrite
		\begin{align}\label{eqn:H'}
			H'(r) & = \left( \dfrac{n-1}{r} + O(M + \delta/r) \right) H(r) + 2 \int_{\partial B_r} u \langle A \nabla u, \nu \rangle. 
		\end{align}
		With the estimates \eqref{eqn:D'} and \eqref{eqn:H'}, we substitute into \eqref{eqn:logN'} to obtain
		\begin{align*}
			\dfrac{d}{dr} \left( \log N(r) \right) = 2 \left( \dfrac{\int_{\partial B_r} \mu^{-1} \langle A \nabla u, \nu \rangle^2 }{ \int_{\partial B_r} u \langle A \nabla u, \nu \rangle }  - \dfrac{  \int_{\partial B_r} u \langle A \nabla u, \nu \rangle }{  \int_{\partial B_r} u^2 \mu   }\right) + O(M + \delta/r),
		\end{align*}
		and conclude the proof of the Lemma by Cauchy-Schwarz. 
	\end{proof}

	\begin{rmk}
		Although they are not the most important parts of Lemma \ref{lem:almost_monoton}, we record for future use the estimates shown above that the vector field $\beta(x)$ and conformal factor $\mu(x)$ satisfy
		\begin{align*}
			\abs{ A \nu - \mu \nu } \lesssim \delta, \, \, \abs{ \divv_{\partial B_r} (A \nu - \mu \nu)   } \lesssim M + \delta/r 
		\end{align*}
		on $\partial B_r$ under the assumptions of Lemma \ref{lem:almost_monoton}.
	\end{rmk}

	With our main tools collected, let us first provide a proof of the unique continuation principle for solutions $u$ that solve \eqref{eqn:soln} in $B_1$ whenever $A$ has an Osgood modulus of continuity $\omega$ with some extra structure. To this end, we continue directly with the growth properties of the frequency function $N_u^A$ defined in Definition \ref{defn:almgren}. Our goal is to show that $N_u^A(r) < \infty$ for all $r \in (0,1)$ by appealing to Corollary \ref{cor:growth_disc}. First, we collect some estimates regarding a non-trivial solution $u$ and its distance to solutions with smooth coefficients.

	In the remainder of this section, we shall frequently assume:
	\begin{enumerate}[({A}1)]
		\item $u$ is a non-trivial solution of \eqref{eqn:soln}, \label{cond:A1}
		\item $A$ is a symmetric, uniformly elliptic matrix with constant $\lambda >0$, and $A(0) = I$, \label{cond:A2}
		\item $A$ has modulus of continuity $\omega$ which satisfies the Osgood condition \eqref{eqn:osgood}, \label{cond:A3}
		\item $0 < r < 1$ is fixed and $0 < N_0 \le N_u^A(r) < \infty$, \label{cond:A4}
	\end{enumerate}
	where $N_0  \gg 1$ is a large constant depending just on the dimension $n$ and $\omega$. As in Proposition \ref{prop:mod}, we write $\phi(s) \coloneqq \omega(s)/s$ and $\psi(s) = \phi(1/s)$. Furthermore, we fix $\eta \in C_c^\infty(B_1)$ to be a smooth bump function with $\int \eta \,dx = 1$.

	\begin{lemma}\label{lem:approx_v}
		Assume \ref{cond:A1}-\ref{cond:A4}, and suppose that $0 < \epsilon < r/2< 1/4$. Let $\ol{A} \coloneqq A * \eta_\epsilon$, and let $v$ be the solution of 
		\begin{equation}\label{eqn:approx_soln}
			\begin{alignedat}{3}
			-\divv(\ol{A} \, \nabla v) & = 0 && \text{ in } B_r,\\
			v & = u && \text{ on } \partial B_r.
			\end{alignedat}
		\end{equation}
		Then for some constant $C_1 >0$ depending just on $n$, $\lambda$, and the choice of $\eta$ we have:
		\begin{align*}
			\fint_{B_r} \abs{ \nabla (u-v) }^2 \, dx & \le C_1 \omega(\epsilon)^2 \fint_{B_r} \abs{\nabla u}^2\, dx , \\
			\fint_{\partial B_{tr}} (u-v)^2 \, d\sigma & \le C_1 (1-t)  \omega(\epsilon)^2 r^2 \fint_{B_r} \abs{\nabla u}^2 \, dx, \, \, tr \in (r(1-\epsilon), r).
		\end{align*} 
	\end{lemma}
	
	\begin{proof}
		Since $A$ is uniformly elliptic with constant $\lambda >0$, then we readily see that $\ol{A}$ is also uniformly elliptic and symmetric with the same constant. Applying Lemmas \ref{lem:approx_id} \ref{lem:qst2_fix}, we immediately obtain the estimate
		\begin{align}\label{eqn:gradient_approx}
			\fint_{B_r} \abs{\nabla(u-v)}^2 \, dx \lesssim \omega(\epsilon)^2 \fint_{B_r} \abs{\nabla u}^2 \, dx 
		\end{align}
		with implicit constant depending just on the dimensions and $\lambda$. As for the second estimate, notice that if $x \in \partial B_r$, and $t \in (1- \epsilon ,1)$, then we can estimate by the fundamental theorem of calculus and the fact that $u \equiv v$ on $\partial B_r$, 
		\begin{align*}
			\abs{(u(tx) - v(tx))}^2 \le \left( \int_{\gamma_{x,t}} \abs{ \nabla(u-v)  } \, d\HD^1\right)^2 \lesssim \HD^{1}(\gamma_{x,t}) \int_{\gamma_{x,t}} \abs{ \nabla (u-v)}^2 \, d\HD^1,
		\end{align*}
		where $\gamma_{x,t} = [tx, x]$ is the segment connecting $tx$ and $x$. Integrating the above for $x \in \partial B_r$, recalling that $t \simeq 1$, and applying the coarea formula leaves us with
		\begin{align*}
			\int_{\partial B_{tr}} (u-v)^2 \, d\sigma \lesssim (1-t)r \int_{B_r} \abs{ \nabla(u-v) }^2 \, dx,
		\end{align*}
		which, after combining with \eqref{eqn:gradient_approx}, gives the second desired estimate.
	\end{proof}
	
	With the previous estimate in hand, we can prove the following dichotomy: either the frequency function is bounded at scale $r$, or otherwise it does not grow too quickly from scale $r$ to scale $r(1-1/N(r))$.

	\begin{lemma}\label{lem:dichot}
		Assume the hypotheses \ref{cond:A1}-\ref{cond:A4} and write $N_u^A(t) = N(t)$. Then the following alternative holds:
		\begin{align*}
			N(r) < N_0, \text{ or } N(  r (1 - s/(N(r) ) \le N(r) + C_1 r \psi(N(r)/r), \, \, s \in [0,1],
		\end{align*}
		where $C_1, N_0 > 0$ are constants depending only on the dimension and the ellipticity constant $\lambda$.
	\end{lemma}
	\begin{proof}
		Let $r \in (0,1)$ be fixed, and set $N \coloneqq N_u^A(r)$. We need only consider the case $N > N_0$, where $N_0$ is large and fixed. Since the hypothesis and conclusion of the Lemma are invariant under multiplication by scalars, we may normalize $u$ so that 
		\begin{align*}
			r^2 \fint_{B_r} \langle A \nabla u, \nabla u \rangle \, dx =  \dfrac{\sigma(\partial B_1)}{\abs{B_1}} N, \qquad   \fint_{\partial B_r} \mu_A u^2 \, d\sigma = 1,
		\end{align*}
		where $\mu_A(x) = \langle A(x) x/\abs{x}, x/\abs{x} \rangle$.
		
		Define $\epsilon = N^{-1}r$. As in Lemma \ref{lem:approx_v}, we let $v$ be a solution of \eqref{eqn:approx_soln}, and denote by $N_v^{\ol{A}}$ the frequency function corresponding to $v$. If we denote by $\mu_{\ol{A}} = \langle \ol{A}(x) x/\abs{x}, x/\abs{x} \rangle$, then Lemma \ref{lem:approx_id} readily gives
		\begin{align}\label{eqn:ola_approx}
			\abs{A - \ol{A}} + \abs{\mu_A(x) - \mu_{\ol{A}}(x)} \lesssim \omega(\epsilon) \ll 1, \, \, \abs{\nabla \ol{A}} \lesssim \epsilon^{-1} \omega(\epsilon) = \phi(\epsilon),
		\end{align}
		and thus from the normalization of $u$ and Lemma \ref{lem:approx_v} it is straight-forward to verify
		\begin{align}\label{eqn:transfer1}
			N_v^{\ol{A}}(r) \le N( 1 + C \omega(\epsilon)).
		\end{align}
		For example, we may estimate
		\begin{align*}
			 \int_{B_r} \langle \ol{A} \nabla v, \nabla v\rangle  & \le C \abs{A- \ol{A}}_{L^\infty} \int_{B_r} \abs{\nabla v}^2  + \int_{B_r} \langle A \nabla v, \nabla v \rangle \\
			 & \le C   \omega(\epsilon) \int_{B_r} \abs{\nabla v}^2 + C \norm{\nabla(u-v)}_{L^2(B_r)} \left( \norm{\nabla u}_{L^2 (B_r)} + \norm{\nabla v}_{L^2 (B_r)} \right)  \\
			  & \qquad \qquad \qquad \qquad  + \int_{B_r} \langle A \nabla u , \nabla u \rangle, \\
			  & \le C \omega(\epsilon) \left( \int_{B_r} \abs{\nabla v}^2 +\abs{\nabla u}^2 \right) + \int_{B_r} \langle A \nabla u , \nabla u \rangle,
		\end{align*}
		so that 
		\begin{align}\label{eqn:just_1}
			r^2 \fint_{B_r} \langle \ol{A} \nabla v, \nabla v \rangle \le r^{2} \fint_{B_r} \langle A \nabla u, \nabla u \rangle + C N \omega(\epsilon).
		\end{align}
		A similar computation shows that for $r_s = r(1-s/N)$ with $s \in [0,1]$,
		\begin{align}\label{eqn:just_2}
			\abs{ \fint_{\partial B_{r_s}} \mu_{\ol{A }} \, v^2 - \mu_{A} u^2 } \le C \omega(\epsilon), 
		\end{align}
		and \eqref{eqn:transfer1} is just a consequence of \eqref{eqn:just_1} and \eqref{eqn:just_2} for $s = 0$.

		Notice that $\abs{\ol{A} - I} \lesssim \omega(2r)$, since $A(0) = I$, and thus Lemma \ref{lem:almost_monoton} applied to the solution $v$ tells us 
		\begin{align*}
			\dfrac{d}{dt} \log\left(  N_v^{\ol{A}}(t)   \right) \ge - C( \phi(\epsilon) + \omega(2r)/t) \ge -C \phi(\epsilon), \, \, t \in (r-\epsilon, r),
		\end{align*}
		since $\phi(s) = \omega(s)/s$ is a decreasing function. Integrating the above from $r_s = r ( 1-s/N)$ to $r$ and applying \eqref{eqn:transfer1} gives (provided that $N_0$ is large)
		\begin{align}\label{eqn:transfer2}
			N_{v}^{\ol{A}}(r_s) \le  e^{ C \epsilon \phi(\epsilon) } N_v^{\ol{A}}(r) \le N_v^{\ol{A}}(r) (1 + C \omega(\epsilon)) \le N ( 1 + C \omega(\epsilon)).
		\end{align}
		Finally, using the same estimates from Lemma \ref{lem:approx_v} to obtain \eqref{eqn:transfer1} gives us in a similar fashion
		\begin{align*}
			N_u^A(r_s) & = \dfrac{(r_s)^2 \fint_{B_{r_s}} \langle A \nabla u, \nabla u \rangle }{  \fint_{\partial B_{r_s}} \mu_A u^2  } \\
			& \le \dfrac{(r_s)^2 \fint_{B_{r_s}} \langle \ol{A} \nabla v, \nabla v \rangle }{  \fint_{\partial B_{r_s}} \mu_{\ol{A}} \,  v^2  } + C \omega(\epsilon)N \\
			& = N_v^{\ol{A}}(r_s) + C N \omega(\epsilon) \le N( 1 + C \omega(\epsilon)).
		\end{align*}
		Recalling that $\epsilon = r/N$ and $\omega(\epsilon) = \epsilon \psi(\epsilon^{-1})$  gives the desired inequality.
	\end{proof}

	As a Corollary of Lemma \ref{lem:dichot} and Corollary \ref{cor:growth_disc}, we prove the boundedness of the frequency function. We shall also see in the proof that such a bound on $N_u^A(r)$ also readily gives a bound on  the doubling properties of $r \ra \fint_{\partial B_r} u^2 \, d\sigma$.
	
	\begin{thm}\label{thm:freq_bdd}
		Assume the hypotheses \ref{cond:A1}-\ref{cond:A3}, and assume in addition that $\omega$ is such that for $\phi(s) = \omega(s)/s$, there is a constant $C_M > 0$ and $g \in L^1((0,\infty); \R_+)$ satisfying
		\begin{align}
			\phi(ts) & \le C_M \phi(t) g(s),  \, \, t,s \in (0,1), \label{eqn:special_mod} \\
			g(t) & \le C_M g(\gamma t), \, \, t \in (0,1), \gamma \in (1/2, 1). \label{eqn:special_mod2}
		\end{align} 
		Then we have the uniform boundedness
		\begin{align*}
			\sup_{0 < r \le 1} N_u^A(r) & \le N < \infty, \\
			\fint_{\partial B_r } u^2 \, d\sigma & \le 2^{N} \fint_{\partial B_{r/2}} u^2 \, d\sigma, \, \, 0 < r \le 1,
		\end{align*}
		where $N$ is a constant depending just on the dimension $n$, ellipticity constant $\lambda$ associated to $A$, $\omega$, and $N_u^A(1)$. 
	\end{thm}
	\begin{proof}
		Take $N_0 \gg 1$ large depending just on $\lambda$ and $n$ so that Lemma \ref{lem:dichot} is satisfied, and write $N(r) \equiv N_u^A(r)$. If we define the function
		\begin{align*}
			\tilde{N}(t) \coloneqq \sup_{r \in (t, 1)} N(t),
		\end{align*}
		then clearly $\tilde{N}$ is decreasing, and if $\tilde{N}(t) \le N_0$ for all $t \in (0,1)$, then the conclusion of the Theorem holds trivially. On the other hand, if $r_0 \in (0,1]$ is the largest value for which $\tilde{N}(r_0) \le N_0$, then it is easy to check that $\tilde{N}(t)$ satisfies the hypothesis of Corollary \ref{cor:growth_disc} for all $t \in (0, r_0)$. Indeed, we need only consider the values $t \in (0, r_0]$ for which $\tilde{N}(t) = N(t)$ (since otherwise, we can decrease $t$ without changing $\tilde{N}(t)$, and then use that $\tilde{N}$ is decreasing). For such values of $t$, Lemma \ref{lem:dichot} and the assumption \eqref{eqn:special_mod} give us that 
		\begin{align*}
			N(t(1-s/N(t))) \le N(t) + C_1 t \psi(N(t)/t) \le N(t) + C_1 C_M t \psi(N(t)) g(t) , \, \, s \in [0,1],
		\end{align*}
		and thus by definition of $\tilde{N}$ and the fact that $N(t) = \tilde{N}(t)$, 
		\begin{align*}
			\tilde{N}(t(1 - 1/ \tilde{N}(t))) \le \tilde{N}(t) + C_1C_M t \psi(\tilde{N}(t)) g(t).
		\end{align*}
		We may then apply Corollary \ref{cor:growth_disc} with $g$, since by assumption $g$ is integrable and doubling. This gives us
		\begin{align*}
			\sup_{t \in (0,1)} \tilde{N}(t) = \sup_{t \in (0,1)} N(t) \eqqcolon \ol{N} < \infty,
		\end{align*}
		which completes the proof of the boundedness of $N_u^A$. 
		
		As for the second claim, we appeal to the same approximation argument as in Lemma \ref{lem:dichot}. Fix any $r \in (0, 1]$, and now define $\epsilon \coloneqq r/(\ol{N} + M)$ for $M \gg 1$ so that $\epsilon$ is small. Normalize $u$ just as before, so that 
		\begin{align*}
			r^2 \fint_{B_r} \langle A \nabla u, \nabla u \rangle  \, dx  = \dfrac{\sigma(\partial B_1)}{\abs{B_1}} N_u^A(r), \qquad \fint_{\partial B_r} \mu_Au^2 \, d\sigma  = 1.
		\end{align*}
		If we let $v$ be the solution of \eqref{eqn:approx_soln}, then the same arguments as in Lemma \ref{lem:dichot} imply that the estimates \eqref{eqn:ola_approx} hold, and since $\ol{N} \ge N_u^A(r)$, we still obtain the estimates
		\begin{align}\label{eqn:doubling_cor}
			N_v^{\ol{A}}(r_s) \le N_u^A(r)(1 + C \omega(\epsilon)), \, \,  \abs{ \fint_{\partial B_{r_s}}  \mu_{\ol{A}} v^2 - \mu_A u^2    } \le C \omega(\epsilon)
		\end{align}
		for $ s \in [0,1]$ where $r_s = r(1-s/(\ol{N} + M))$. From the second conclusion of Lemma \ref{lem:almost_monoton} and the bound above, we see
		\begin{align*}
			\log \left( \dfrac{ r^{-n+1} H_v^{\ol{A}}(r)    }{  r_s^{-n+1} H_v^{\ol{A}}(r)  }  \right) & \le \int_{r_s}^r \dfrac{2 N_v^{\ol{A}}(t)}{t} + e(t) \, dt \\
			& \le 2N_u^{A}(r)(1 + C \omega(\epsilon)) \log(r/r_s) + C \phi(\epsilon) (r-r_s) \\
			& \le \dfrac{4 N_u^A(r)(1 + C \omega(\epsilon))}{\ol{N} + M } + C \epsilon \phi(\epsilon) \le 4  + C \omega(\epsilon) \le 5,
		\end{align*}
		as long as $M$ is large enough. In other words, we have shown that 
		\begin{align*}
			\fint_{\partial B_r} \mu_{\ol{A}} \, v^2 \le e^5 \fint_{\partial B_{r_s}} \mu_{\ol{A}} \, v^2,
		\end{align*}
		and so finally from the second estimate in \eqref{eqn:doubling_cor}, we see that 
		\begin{align*}
			\fint_{\partial B_r} \mu_A u^2 \le e^6 \fint_{\partial B_{r_s}} \mu_A u^2
		\end{align*}
		as long as $M$ is large enough, for arbitrary $s \in [0,1]$. Since the ratio $r/r_1$ is bounded away from $1$ (by a constant depending on $\ol{N}$ and $M$), the second conclusion of the Theorem readily follows; if $k_0 \in \N$ is the smallest number so that $(1-1/(\ol{N} + M))^{k_0} \le 1/4$, then applying the above inequality at most $ k_0$ times, we obtain
		\begin{align*}
			\fint_{\partial B_{r}} \mu_A u^2 \le e^{ 6 {k_0} } \fint_{\partial B_{r/2}} \mu_A u^2.
		\end{align*}
		Recalling that $\mu_A \simeq 1$, yields the doubling inequality as claimed.
	\end{proof}

	\section{Appendix}
	
	Here we list some elementary estimates for solutions of uniformly elliptic equations. The first is a basic approximate-identity computation which we omit.
	
	\begin{lemma}\label{lem:approx_id}
		Let $f: B_r \subset \R^n \ra \R$ have modulus of continuity $\omega$, and let $\eta \in C^\infty_c(\R^n)$ be a nonnegative function satisfying $\spt \eta \subset B_{1}$, and $\int \eta \, dx= 1$. Writing $f_\epsilon = f * \eta_\epsilon$ for $\epsilon \in (0,r)$, we have the estimates
		\begin{align*}
			\abs{f_\epsilon(x) - f(x)} \le \omega(\epsilon), \, \, 	\abs{\nabla f_\epsilon(x)} \le C \epsilon^{-1} \omega(\epsilon), \, \, x \in B_{r - \epsilon},
		\end{align*}
		where $C$ depends just on the dimension $n$ and $\int  \abs{ \nabla \eta} \, dx $. 
	\end{lemma}

	Next, we have a simple quantitative stability for solutions of elliptic equations, which for completeness we prove. First, we remind the reader of the notation
	\begin{align*}
		\abs{\nabla_A u}^2 \coloneqq A \nabla u \cdot \nabla u,
	\end{align*}
	whenever $A$ is uniformly elliptic and symmetric.

	\begin{lemma}\label{lem:qst2_fix}
		Let $A_0, A_1$ be two uniformly elliptic, symmetric matrices (both with ellipticity constant $\lambda_0 \in (0,1)$). Let $u_0 = u_1$ on $\partial B_1$ and $-\divv(A_i \nabla u_i) =0$ in $B_1$ for $i=0,1$, and suppose that for some $r \in (0,1)$, we have that 
		\begin{enumerate}[(a)]
			\item $\abs{A_0 - A_1} \le \epsilon$ in $B_1 \setminus B_r$, \label{cond:qst1_fix}
			\item $\int_{B_r} \abs{\nabla u_0}^2 \le \delta  \int_{B_1} \abs{\nabla u_0}^2$, \label{cond:qst2_fix}
		\end{enumerate}
		for $\epsilon, \delta \in (0,1)$ sufficiently small. Then 
		\begin{align*}
			\int_{ B_1} \abs{ \nabla (u_0 - u_1) }^2 \le C_1 \max\{\epsilon^2, \delta \} \min \left \{  \int_{B_1} \abs{ \nabla u_0}^2, \int_{B_1} \abs{ \nabla u_1}^2 \right \}.
		\end{align*}
		with constant $C_1 >0$ depending on the dimension $n$ and $\lambda$.
	\end{lemma}
	
	\begin{proof}
		Since $u_0 = u_1$ on $\partial B_1$, it is straight-forward to see via energy minimization and ellipticity of $A_0, A_1$ that
		\begin{align*}
			\int_{B_1} \abs{\nabla_{A_0} u_0}^2 \le \int_{B_1} \abs{\nabla_{A_0} u_1}^2 \le \lambda_0^{-2} \int_{B_1} \abs{\nabla_{A_1} u_1}^2, \\
			\int_{B_1} \abs{\nabla_{A_1} u_1}^2 \le \int_{B_1} \abs{\nabla_{A_1} u_0}^2 \le \lambda_0^{-2} \int_{B_1} \abs{\nabla_{A_0} u_0}^2,
		\end{align*}
		so that $\int_{B_1} \abs{\nabla u_0}^2 \simeq_{\lambda_0} \int_{B_1} \abs{\nabla u_1}^2$. Using the weak formulation for $u_1$ and then $u_0$, we also see that 
		\begin{align*}
			\int_{B_1} A_1 \nabla (u_0 - u_1) &  \cdot \nabla(u_0 - u_1) \\
			& =\int_{B_1} A_1 \nabla u_0 \cdot \nabla (u_0 - u_1)\\
			& = \int_{B_1} (A_1 - A_0) \nabla u_0 \cdot \nabla (u_0 - u_1) \\
			&  \le \epsilon \int_{B_1 \setminus B_r} \abs{\nabla u_0} \abs{ \nabla (u_0 - u_1)} + 2 \lambda^{-1} \int_{B_r} \abs{\nabla u_0} \abs{ \nabla (u_0 - u_1)} \\
			& \le \epsilon \left( \int_{B_1} \abs{\nabla u_0}^2 \right)^{1/2} \left( \int_{B_1 } \abs{\nabla (u_0 - u_1)}^2 \right)^{1/2} \\
			& \qquad + 2 \lambda^{-1} \delta^{1/2} \left( \int_{B_1} \abs{\nabla u_0}^2 \right)^{1/2} \left( \int_{B_1} \abs{\nabla (u_0 - u_1) }^2 \right)^{1/2}.
		\end{align*}
		Using the ellipticity of $A_1$, we conclude that
		\begin{align*}
			\int_{B_1} \abs{ \nabla (u_0 - u_1)}^2 & \le \lambda^{-1} \int_{B_1} A_1 \nabla (u_0 - u_1) \cdot \nabla (u_0 - u_1) \\
			&  \le (\epsilon + 2\lambda^{-1} \delta^{1/2}) \left( \int_{B_1} \abs{\nabla u_0}^2 \right)^{1/2} \left( \int_{B_1} \abs{\nabla (u_0 - u_1) }^2 \right)^{1/2},
		\end{align*}
		and so rearranging then allows us to conclude that 
		\begin{align*}
			\int_{B_1} \abs{\nabla (u_0 - u_1)}^2 \le \lambda^{-2} \left( \epsilon^2 + 4 \lambda^{-1} \epsilon \delta^{1/2} + 4 \lambda^{-2} \delta  \right) \int_{B_1} \abs{\nabla u_0}^2,
		\end{align*}
		and the Lemma then follows.
	\end{proof}

\bibliographystyle{alpha}
\bibliography{bibl}

\newcommand{\etalchar}[1]{$^{#1}$}
\begin{thebibliography}{DCFL{\etalchar{+}}17}

\bibitem[Agm66]{Agmon}
Shmuel Agmon.
\newblock {\em Unicit\'e{} et convexit\'e{} dans les probl\`emes
  diff\'erentiels}, volume No. 13 (\'Et\'e, 1965) of {\em S\'eminaire de
  Math\'ematiques Sup\'erieures [Seminar on Higher Mathematics]}.
\newblock Les Presses de l'Universit\'e{} de Montr\'eal, Montreal, QC, 1966.

\bibitem[Aro57]{Aronszajn}
N.~Aronszajn.
\newblock A unique continuation theorem for solutions of elliptic partial
  differential equations or inequalities of second order.
\newblock {\em J. Math. Pures Appl. (9)}, 36:235--249, 1957.

\bibitem[BK05]{BK05}
Jean Bourgain and Carlos~E. Kenig.
\newblock On localization in the continuous {A}nderson-{B}ernoulli model in
  higher dimension.
\newblock {\em Invent. Math.}, 161(2):389--426, 2005.

\bibitem[BN55]{BN55}
L.~Bers and L.~Nirenberg.
\newblock On a representation theorem for linear elliptic systems with
  discontinuous coefficients and its applications.
\newblock In {\em Convegno {I}nternazionale sulle {E}quazioni {L}ineari alle
  {D}erivate {P}arziali, {T}rieste, 1954}, pages 111--140. Ed. Cremonese, Rome,
  1955.

\bibitem[Car39]{Carleman-planar}
T.~Carleman.
\newblock Sur un probl\`eme d'unicit\'e{} pur les syst\`emes d'\'equations aux
  d\'eriv\'ees partielles \`a{} deux variables ind\'ependantes.
\newblock {\em Ark. Mat. Astr. Fys.}, 26(17):9, 1939.

\bibitem[CLMM24]{CLMM24}
Sagun Chanillo, Alexander Logunov, Eugenia Malinnikova, and Dan Mangoubi.
\newblock Local version of {C}ourant's nodal domain theorem.
\newblock {\em J. Differential Geom.}, 126(1):49--63, 2024.

\bibitem[CNV15]{CNV15}
Jeff Cheeger, Aaron Naber, and Daniele Valtorta.
\newblock Critical sets of elliptic equations.
\newblock {\em Comm. Pure Appl. Math.}, 68(2):173--209, 2015.

\bibitem[CR16]{CR16}
Pedro Caro and Keith~M. Rogers.
\newblock Global uniqueness for the {C}alder\'on problem with {L}ipschitz
  conductivities.
\newblock {\em Forum Math. Pi}, 4:e2, 28, 2016.

\bibitem[CS07]{CS07}
Luis Caffarelli and Luis Silvestre.
\newblock An extension problem related to the fractional {L}aplacian.
\newblock {\em Comm. Partial Differential Equations}, 32(7-9):1245--1260, 2007.

\bibitem[Dav25]{Davey25}
Blair Davey.
\newblock A frequency function approach to quantitative unique continuation for
  elliptic equations.
\newblock Preprint, {arXiv}:2506.19130 [math.{AP}] (2025), 2025.

\bibitem[DCFL{\etalchar{+}}17]{DCFLVW17}
M.~Di~Cristo, E.~Francini, C.-L. Lin, S.~Vessella, and J.-N. Wang.
\newblock Carleman estimate for second order elliptic equations with
  {L}ipschitz leading coefficients and jumps at an interface.
\newblock {\em J. Math. Pures Appl. (9)}, 108(2):163--206, 2017.

\bibitem[DF88]{DF88}
Harold Donnelly and Charles Fefferman.
\newblock Nodal sets of eigenfunctions on {R}iemannian manifolds.
\newblock {\em Invent. Math.}, 93(1):161--183, 1988.

\bibitem[EJS25]{EJSpreprint}
Max Engelstein, Cole Jeznach, and Yannick Sire.
\newblock Singular set estimates for solutions to elliptic equations in higher
  co-dimension.
\newblock Preprint, {arXiv}:2502.03294 [math.{AP}] (2025), 2025.

\bibitem[ESS03]{ESS03b}
L.~Escauriaza, G.~A. Ser\"egin, and V.~Shverak.
\newblock {$L_{3,\infty}$}-solutions of {N}avier-{S}tokes equations and
  backward uniqueness.
\newblock {\em Uspekhi Mat. Nauk}, 58(2(350)):3--44, 2003.

\bibitem[ESvS03]{ESS03a}
L.~Escauriaza, G.~Seregin, and V.~\v~Sver\'ak.
\newblock Backward uniqueness for parabolic equations.
\newblock {\em Arch. Ration. Mech. Anal.}, 169(2):147--157, 2003.

\bibitem[Fos25]{Foster25}
Benjamin Foster.
\newblock Results on gradients of harmonic functions on {L}ipschitz surfaces.
\newblock {\em Calc. Var. Partial Differential Equations}, 64(7):Paper No. 199,
  2025.

\bibitem[FVW22]{FVW22}
Elisa Francini, Sergio Vessella, and Jenn-Nan Wang.
\newblock Carleman estimate for complex second order elliptic operators with
  discontinuous {L}ipschitz coefficients.
\newblock {\em J. Spectr. Theory}, 12(2):535--571, 2022.

\bibitem[Gal23]{Gallegos23}
Josep~M. Gallegos.
\newblock Size of the zero set of solutions of elliptic {PDE}s near the
  boundary of {L}ipschitz domains with small {L}ipschitz constant.
\newblock {\em Calc. Var. Partial Differential Equations}, 62(4):Paper No. 113,
  52, 2023.

\bibitem[GL86]{GL86}
Nicola Garofalo and Fang-Hua Lin.
\newblock Monotonicity properties of variational integrals, {$A_p$} weights and
  unique continuation.
\newblock {\em Indiana Univ. Math. J.}, 35(2):245--268, 1986.

\bibitem[GL87]{GL87}
Nicola Garofalo and Fang-Hua Lin.
\newblock Unique continuation for elliptic operators: a geometric-variational
  approach.
\newblock {\em Comm. Pure Appl. Math.}, 40(3):347--366, 1987.

\bibitem[GSVG14]{GSVG14}
Nicola Garofalo and Mariana Smit Vega~Garcia.
\newblock New monotonicity formulas and the optimal regularity in the
  {S}ignorini problem with variable coefficients.
\newblock {\em Adv. Math.}, 262:682--750, 2014.

\bibitem[Hab15]{Haberman15}
Boaz Haberman.
\newblock Uniqueness in {C}alder\'on's problem for conductivities with
  unbounded gradient.
\newblock {\em Comm. Math. Phys.}, 340(2):639--659, 2015.

\bibitem[HJ23]{HW23-holder}
Yiqi Huang and Wenshuai Jiang.
\newblock Volume {Estimates} for {Singular} sets and {Critical} {Sets} of
  {Elliptic} {Equations} with {H{\"o}lder} {Coefficients}.
\newblock Preprint, {arXiv}:2309.08089 [math.{AP}] (2023), 2023.

\bibitem[JK85]{JK85}
David Jerison and Carlos~E. Kenig.
\newblock Unique continuation and absence of positive eigenvalues for
  {S}chr\"odinger operators.
\newblock {\em Ann. of Math. (2)}, 121(3):463--494, 1985.
\newblock With an appendix by E. M. Stein.

\bibitem[KSU07]{KSU07}
Carlos~E. Kenig, Johannes Sj\"ostrand, and Gunther Uhlmann.
\newblock The {C}alder\'on problem with partial data.
\newblock {\em Ann. of Math. (2)}, 165(2):567--591, 2007.

\bibitem[KSW15]{KSW15}
Carlos Kenig, Luis Silvestre, and Jenn-Nan Wang.
\newblock On {L}andis' conjecture in the plane.
\newblock {\em Comm. Partial Differential Equations}, 40(4):766--789, 2015.

\bibitem[KT01]{KT01}
Herbert Koch and Daniel Tataru.
\newblock Carleman estimates and unique continuation for second-order elliptic
  equations with nonsmooth coefficients.
\newblock {\em Comm. Pure Appl. Math.}, 54(3):339--360, 2001.

\bibitem[LM18a]{LM18-23}
Alexander Logunov and Eugenia Malinnikova.
\newblock Nodal sets of {L}aplace eigenfunctions: estimates of the {H}ausdorff
  measure in dimensions two and three.
\newblock In {\em 50 years with {H}ardy spaces}, volume 261 of {\em Oper.
  Theory Adv. Appl.}, pages 333--344. Birkh\"auser/Springer, Cham, 2018.

\bibitem[LM18b]{LM18}
Alexander Logunov and Eugenia Malinnikova.
\newblock Quantitative propagation of smallness for solutions of elliptic
  equations.
\newblock In {\em Proceedings of the {I}nternational {C}ongress of
  {M}athematicians---{R}io de {J}aneiro 2018. {V}ol. {III}. {I}nvited
  lectures}, pages 2391--2411. World Sci. Publ., Hackensack, NJ, 2018.

\bibitem[LM20]{LM-notes}
Alexander Logunov and Eugenia Malinnikova.
\newblock Lecture notes on quantitative unique continuation for solutions of
  second order elliptic equations.
\newblock In {\em Harmonic analysis and applications}, volume~27 of {\em
  IAS/Park City Math. Ser.}, pages 1--33. Amer. Math. Soc., [Providence], RI,
  [2020] \copyright 2020.

\bibitem[LMNN21]{LMNN21}
A.~Logunov, E.~Malinnikova, N.~Nadirashvili, and F.~Nazarov.
\newblock The sharp upper bound for the area of the nodal sets of {D}irichlet
  {L}aplace eigenfunctions.
\newblock {\em Geom. Funct. Anal.}, 31(5):1219--1244, 2021.

\bibitem[LMNN25]{LMNN25}
A.~Logunov, E.~Malinnikova, N.~Nadirashvili, and F.~Nazarov.
\newblock The {L}andis conjecture on exponential decay.
\newblock {\em Invent. Math.}, 241(2):465--508, 2025.

\bibitem[Log18a]{Logunov18-poly}
Alexander Logunov.
\newblock Nodal sets of {L}aplace eigenfunctions: polynomial upper estimates of
  the {H}ausdorff measure.
\newblock {\em Ann. of Math. (2)}, 187(1):221--239, 2018.

\bibitem[Log18b]{Logunov18-lower}
Alexander Logunov.
\newblock Nodal sets of {L}aplace eigenfunctions: proof of {N}adirashvili's
  conjecture and of the lower bound in {Y}au's conjecture.
\newblock {\em Ann. of Math. (2)}, 187(1):241--262, 2018.

\bibitem[LRL13]{LRL13}
J\'er\^ome Le~Rousseau and Nicolas Lerner.
\newblock Carleman estimates for anisotropic elliptic operators with jumps at
  an interface.
\newblock {\em Anal. PDE}, 6(7):1601--1648, 2013.

\bibitem[Man98]{Mandache-ceg}
Niculae Mandache.
\newblock On a counterexample concerning unique continuation for elliptic
  equations in divergence form.
\newblock {\em Math. Phys. Anal. Geom.}, 1(3):273--292, 1998.

\bibitem[Mil74]{Miller-ceg}
Keith Miller.
\newblock Nonunique continuation for uniformly parabolic and elliptic equations
  in self-adjoint divergence form with {H}\"older continuous coefficients.
\newblock {\em Arch. Rational Mech. Anal.}, 54:105--117, 1974.

\bibitem[MS69]{MS69}
A.~Marino and S.~Spagnolo.
\newblock Un tipo di approssimazione dell'operatore {$\sum_1^nij D\sb
  i(a_{ij}(x)D_j)$} con operatori {$\sum_1^njD_j(\beta (x)D\sb j)$}.
\newblock {\em Ann. Scuola Norm. Sup. Pisa Cl. Sci. (3)}, 23:657--673, 1969.

\bibitem[NV17a]{NV17-annals}
Aaron Naber and Daniele Valtorta.
\newblock Rectifiable-{R}eifenberg and the regularity of stationary and
  minimizing harmonic maps.
\newblock {\em Ann. of Math. (2)}, 185(1):131--227, 2017.

\bibitem[NV17b]{NV17}
Aaron Naber and Daniele Valtorta.
\newblock Volume estimates on the critical sets of solutions to elliptic
  {PDE}s.
\newblock {\em Comm. Pure Appl. Math.}, 70(10):1835--1897, 2017.

\bibitem[Pli63]{Plis-ceg}
A.~Pli\'s.
\newblock On non-uniqueness in {C}auchy problem for an elliptic second order
  differential equation.
\newblock {\em Bull. Acad. Polon. Sci. S\'er. Sci. Math. Astronom. Phys.},
  11:95--100, 1963.

\bibitem[Sch98]{Schulz98}
Friedmar Schulz.
\newblock On the unique continuation property of elliptic divergence form
  equations in the plane.
\newblock {\em Math. Z.}, 228(2):201--206, 1998.

\bibitem[Spo22]{Spolaor22}
Luca Spolaor.
\newblock Monotonicity formulas in the calculus of variation.
\newblock {\em Notices Amer. Math. Soc.}, 69(10):1731--1737, 2022.

\bibitem[STT20]{SST20}
Yannick Sire, Susanna Terracini, and Giorgio Tortone.
\newblock On the nodal set of solutions to degenerate or singular elliptic
  equations with an application to {$s$}-harmonic functions.
\newblock {\em J. Math. Pures Appl. (9)}, 143:376--441, 2020.

\bibitem[SVG24]{SVG24}
Mariana Smit Vega~Garcia.
\newblock Almgren-type monotonicity formulas.
\newblock {\em Mat. Contemp.}, 61:20--48, 2024.

\bibitem[Tol23]{Tolsa23}
Xavier Tolsa.
\newblock Unique continuation at the boundary for harmonic functions in {$C^1$}
  domains and {L}ipschitz domains with small constant.
\newblock {\em Comm. Pure Appl. Math.}, 76(2):305--336, 2023.

\bibitem[Wol92]{Wolff92}
T.~H. Wolff.
\newblock A property of measures in {${\bf R}^N$} and an application to unique
  continuation.
\newblock {\em Geom. Funct. Anal.}, 2(2):225--284, 1992.

\bibitem[Yu17]{Yu17}
Hui Yu.
\newblock Unique continuation for fractional orders of elliptic equations.
\newblock {\em Ann. PDE}, 3(2):Paper No. 16, 21, 2017.

\end{thebibliography}

\end{document}